\documentclass[twoside,a4paper,11pt]{article}
\usepackage{amsfonts, amsbsy, amsmath, amsthm, amssymb, latexsym}
\usepackage{mathrsfs}
\usepackage{enumerate,centernot}
\usepackage[top=30mm,right=30mm,bottom=30mm,left=30mm]{geometry}

\usepackage{bm}
\usepackage{fancyhdr}

\numberwithin{equation}{section}

\newcommand{\Q}{\mathbb{Q}}
\newcommand{\C}{\mathbb{C}}

\newcommand{\codi}{{{\rm codim} \ }}

\renewcommand{\bf}{\textbf}

 \renewcommand{\to}{\rightarrow}

\newcommand{\SO}{{\rm SO}}
\newcommand{\PGL}{{\rm PGL}}
\newcommand{\Hom}{{\rm Hom}}

\newcommand{\Ad}{\mathrm{Ad}}
\newcommand{\Epi}{\mathrm{Epi}}
\newcommand{\Lie}{\mathrm{Lie}}

\newcommand{\uG}{\underline{G}}
\newcommand{\uH}{\underline{H}}
\newcommand{\uK}{\underline{K}}

\newcommand{\frakg}{\mathfrak{g}}

\newcommand{\uW}{\underline{W}}
\newcommand{\uX}{\underline{X}}
\newcommand{\uY}{\underline{Y}}

\newcommand{\R}{\mathbb{R}}

\newtheorem{thm}{Theorem}[section]

\newtheorem{prop}[thm]{Proposition}
\newtheorem{lem}[thm]{Lemma}
\newtheorem{cor}[thm]{Corollary}

\newtheorem{rmk}[thm]{Remark}

\theoremstyle{definition}

\newtheorem{remk}[thm]{Remark}

\newtheorem*{ack}{Acknowledgments}

\begin{document}

\title{Deformation theory\\ and finite simple quotients of triangle groups II}

\author{Michael Larsen, Alexander Lubotzky, Claude Marion}

\date{\today}
\maketitle

\textbf{Abstract}   This paper is a continuation of our first paper \cite{LLM1} in which we showed how deformation theory of representation varieties can be used to study finite simple quotients of triangle groups. While in Part I, we mainly used deformations of the principal homomorphism from ${\rm SO}(3,\R)$, in this part we use ${\rm PGL}_2(\R)$ as well as deformations of representations which are very different from the principal homomorphism.

\section{Introduction}

This paper is a continuation of \cite{LLM1} where it was shown that deformation theory of representation varieties of finitely generated groups $\Gamma$, and in particular
 of hyperbolic  triangle groups $\Gamma=T$, can be used to prove the existence of many finite simple quotients of 
 $\Gamma$.  Let us recall some basic notation. Let $$T=T_{a,b,c}=\langle x,y,z: x^a=y^b=z^c=xyz=1 \rangle$$ be a hyperbolic triangle group so that $a,b,c\in \mathbb{N}$ satisfy $1/a+1/b+1/c<1$. Without loss of generality, we assume $a\leq b \leq c$ and call $(a,b,c)$ a hyperbolic triple of integers.
 
We let $X$ be an irreducible Dynkin diagram and denote by $X(\mathbb{C})$ (resp. $\Lie(X)$) the simple adjoint algebraic group over $\mathbb{C}$ (resp. the simple complex Lie algebra) of type $X$. Also $X(p^\ell)$ denotes the untwisted finite simple group of type $X$ over $\mathbb{F}_{p^\ell}$.
 We say that $T$ is saturated with finite quotients of type $X$ if there exist $p_0, e \in \mathbb{N}$ such that for all primes  $p>p_0$, $X(p^{e\ell})$ is  a quotient of $T$ for every $\ell \in \mathbb{N}$, and for a set of positive density of primes $p$, we even have $X(p^\ell)$ is a quotient of $T$ for every $\ell \in \mathbb{N}$. 
 
 The main idea of \cite{LLM1} was the observation (see Theorem 4.1 therein) that $T$  is saturated with finite quotients of type $X$ if and only if there exist a simple algebraic group $\uG$ over $\mathbb{C}$ of type $X$ and a Zariski dense representation $\rho: T \rightarrow \uG(\mathbb{C})$ which is not locally rigid, i.e. $\dim H^1(T,\frak{g})>0$, where $\frak{g}$ is the Lie algebra of $\uG$ and $T$ acts on $\frak{g}$ via ${\rm Ad}\circ \rho$.
 
 In \cite{LLM1} we showed that for all pairs $(X,(a,b,c))$ which are not listed in \cite[Table 1]{LLM1}, $T$ is saturated with finite quotients of type $X$. The main goal of the current paper is to push the deformation method further in order to eliminate some of the cases left unsettled in \cite[Table 1]{LLM1}.   
 
In \cite{LLM1} we produced representations  of $T$ into  an absolutely simple compact real form $\uG$ of $X$  by first  using a Zariski dense representation of $T$ into ${\rm SO}(3,\R)$. From there, we  deformed the representation $T\rightarrow {\rm SO}(3,\R) \rightarrow \uG(\mathbb{R})$ induced from the principal homomorphism ${\rm SO}(3) \rightarrow \uG$. This method did not permit us to consider the six triangle groups in
$$S=\{T_{2,4,6},T_{2,6,6},T_{2,6,10},T_{3,4,4},T_{3,6,6},T_{4,6,12}\},$$
which are the (only) hyperbolic triangle groups without ${\rm SO}(3)$-dense representations (see \cite{LL}). 
So our first goal will be to extend in \S \ref{s:nonso3dense} the method we implemented in \cite{LLM1}  for compact forms, to non compact forms. This time we will start with a representation $T\rightarrow {\rm PGL}_2$ instead of $T\rightarrow {\rm SO}(3)$. In this way, our results will also include these six groups. Note that every Fuchsian group admits a Zariski dense embedding into ${\rm PGL}_2(\mathbb{C})$, so this method can be applied to any (hyperbolic) triangle group, at the
cost of some additional complications.
 
In \cite{LLM1} we sometimes use  ``two-step ladders" or even ``three-step ladders" $$T \rightarrow {\rm SO}(3)\rightarrow \uK\rightarrow \uH \rightarrow \uG$$ to deform the representation $T\to \SO(3)\to \uG$
first to a dense homomorphism to $\uK$, thence to a dense homomorphism to $\uH$, and finally to a dense homomorphism to $\uG$.
Here, we use a non-compact version of the same idea.

Some cases which cannot be covered by the principal homomorphism method can still be dealt by variants of the deformation-theoretic approach.  Here we present two such:

\begin{enumerate}[(i)]
\item Starting with a Zariski dense representation of $T$ into a group of type $B_{k-1}\times B_{r-k} \subset D_{r}$
we deform it to a Zariski dense representation into a group of type $D_r$.  Here, the novelty is that the homomorphism
$\PGL_2\to D_r$ is non-principal even though each homomorphism $\PGL_2\to B_i$ is principal.

\item Starting with a representation of $T$ onto the finite group 
$${\rm Alt}_n \subset {\rm SO}(n-1),$$
we deform it to a Zariski dense representation to ${\rm SO}(n-1)$.
\end{enumerate}

Using these methods in \S \ref{s:bibi} and \S \ref{s:altm}, respectively, we will conclude that:

\begin{thm}\label{t:main}
The hyperbolic triangle group $T=T_{a,b,c}$ is saturated with finite quotients of type $X$ except possibly if  $(T,X)$ appears in Table \ref{tab:main} or  Table \ref{tab:mainrigid}.
\end{thm}

\begin{table}
\caption{Possible  (nonrigid) exceptions  to Theorem \ref{t:main}}\label{tab:main}
\center
\begin{tabular}{|l|l|l|}
\hline
$X$& $(a,b,c)$ & $r$\\
\hline
 $A_r$& $(2,3,7)$ & $5 \leq r \leq 19$ \\
& $(2,3,8)$ & $5 \leq r \leq 13$\\
& $(2,3,c)$, $c \geq 9$ & $ 5 \leq r\leq 7$\\
& $(2,4,5)$ & $ 3 \leq r\leq 13$\\
& $(2,4,6)$ & $ 3 \leq r\leq 9$\\
& $(2,4,c)$, $c \geq 7$& $ 3 \leq r \leq 5$\\
& $(2,5,5)$ & $r=6$\\
& $(2,b,c)$, $b \geq 5$, $c \geq 5$ & $r=3$\\
& $(3,3,c)$, $c \geq 4$ & $r\in \{3,4,6\}$\\
\hline
$B_3$ & $(2,3,c)$, $c \geq 7$&   \\
& $(3,3,c)$, $c \geq 4$, $c \neq 15c_1$ & \\
& $(2,4,5)$ & \\
& $(2,5,5)$ & \\
\hline
 $D_r$ & $(2,3,7)$ & $r \in\{4,5,9\}$\\
& $(2,3,8)$& $r\in \{4,5\}$ \\
& $(2,3,9)$ & $r\in \{4,5\}$\\
& $(2,3,10)$ & $r\in \{4,5\}$\\
& $(2,3,c)$, $c\geq 11$, $c\neq 15c_1$ & $r=4$\\
& $(2,3,c)$, $c\geq 12$, $c \neq 11c_1$ & $r=5$\\
& $(2,4,5)$ & $r =5$\\
& $(3,3,4)$ & $r \in\{4,5\}$\\
& $(3,3,c)$, $c\geq 5$ and & $r =4$\\
& $c\not \in \{7c_1,9c_1,10c_1,12c_1,15c_1\}$  & \\
\hline
$E_6$ & $(2,3,7)$ &\\
& $(2,3,8)$& \\
&$(2,4,5)$ & \\
& $(2,4,6)$&\\
& $(2,4,7)$&\\
& $(2,4,8)$ & \\
\hline
\end{tabular}\\
Here $c_1$ denotes any natural number
\end{table}

\begin{table}
\caption{Rigid exceptions  to Theorem \ref{t:main}}\label{tab:mainrigid}
\center
\begin{tabular}{|l|l|}
\hline
$X$& $(a,b,c)$ \\
\hline
 $A_1$& any  \\
$A_2$& $a=2$\\
$A_3$ & $a=2$, $b=3$\\
$A_4$ & $a=2$, $b=3$\\
\hline
$C_2$ & $b=3$\\
\hline
$G_2$ & $a=2$, $c=5$\\
\hline
\end{tabular}
\end{table}

For the cases  appearing in Table \ref{tab:mainrigid} we know for sure that $T$ is not saturated with finite quotients of type $X$. (These are the rigid cases---see \cite{Marionconj} and \cite{LLM1}.)  For the rest (i.e. the  cases appearing in Table \ref{tab:main}) we do not know the answer. 

Examining Tables \ref{tab:main} and \ref{tab:mainrigid} we can immediately deduce:

\begin{cor}\label{cor:first}
The following two assertions hold:
\begin{enumerate}[(i)]
\item  If $\mu=1/a+1/b+1/c\leq1/2$ then for every simple Dynkin diagram $X\neq A_1$, $T_{a,b,c}$ is saturated with finite  quotients of type $X$.
\item Let\begin{eqnarray*}Y& = & \{A_r:1\leq r\leq 19\} \cup \{B_3\}\cup \{C_2\}\cup \{G_2\}\cup\{E_6\}\\
& & \cup\{D_r: r=4,5,9\}.
\end{eqnarray*}
Then for every hyperbolic triple $(a,b,c)$ and every simple Dynkin diagram $X \not \in Y$, $T_{a,b,c}$ is saturated with finite  quotients of type $X$.
\end{enumerate}
\end{cor}

\begin{cor}
Assume $X \not \in \{A_r:1\leq r\leq 7\}\cup \{B_3\}\cup\{C_2\}\cup\{D_r: r=4,5\}$. Then for almost every hyperbolic triple $(a,b,c)$, the group $T=T_{a,b,c}$ is saturated with finite  quotients of type $X$.
\end{cor}

Many of our results are new even in the classical case $(a,b,c)=(2,3,7)$.

\begin{cor}
The triangle group $T_{2,3,7}$ is saturated with finite  quotients of type $X$ for every $X$ which is not in $\{A_r:1\leq r\leq 19\} \cup \{B_3\}\cup \{C_2\}\cup \{D_r:r=4,5,9\}\cup\{E_6\}$. In particular, it is saturated with finite quotients of type $E_8$.
\end{cor}

This answers a question we were asked by Guralnick. In fact, as already seen in Corollary \ref{cor:first}(ii), we have even more.

\begin{cor}
Every hyperbolic triangle group is saturated with finite  quotients of type $E_7$ and $E_8$.
\end{cor}

\begin{ack}
The authors are grateful to the ERC, ISF and NSF for their support.
\end{ack}

\section{Preliminary results}

This section consists  of some preliminary results on deformation theory of hyperbolic triangle groups and on saturation of hyperbolic triangle groups by finite quotients of a given type. For  more details, see \cite{LLM1}.   
 
Let $T=T_{a,b,c}$ be a hyperbolic triangle group and $\uG$ be a simple algebraic group over $\mathbb{C}$ of type $X$. If $\rho \in \Hom(T,\uG(\C))=\Hom(T,\uG)(\C)$, then $T$ acts on the Lie algebra $\frak{g}$ of $\uG$ via ${\rm Ad}\circ \rho$, where ${\rm Ad}:\uG\rightarrow {\rm Aut}(\frak{g})$ denotes the adjoint representation of $\uG$. To avoid confusion we will sometimes write ${\rm Ad}\circ \rho\mid_{\frak{g}}$ for the action of $T$ on $\frak{g}$ via ${\rm Ad}\circ\rho$. We let
$Z^1(T,{\rm Ad}\circ \rho)$ (respectively, $B^1(T,{\rm Ad}\circ\rho)$) be the corresponding  space of 1-cocycles (respectively, 1-coboundaries) and set $$H^1(T,{\rm Ad}\circ \rho)=Z^1(T,{\rm Ad}\circ \rho)/B^1(T,{\rm Ad}\circ \rho).$$

The following result is due to Weil (see \cite{Weil}). In the statement, for $t \in \{x,y,z\}$, $\frak{g}^t$ denotes the fixed point space of $t$ in $\frak{g}$ (under the action ${\rm Ad}\circ \rho$).

\begin{thm}\label{t:weil}
The following assertions hold:
\begin{enumerate}[(i)]
 \item The space $Z^1(T,{\rm Ad}\circ \rho)$ is the Zariski tangent space at $\rho$ in $\Hom(T,\uG)$  and $$\dim Z^1(T,{\rm Ad}\circ \rho)= 2\dim \frak{g}+i^*-(\dim \frak{g}^x+\dim \frak{g}^y+\dim \frak{g}^z)$$
where $i$ and $i^*$ denote the dimensions of the space of invariants of ${\rm Ad}\circ\rho$ and $({\rm Ad}\circ \rho)^*$ on $\frak{g}$ and $\frak{g}^*$, respectively.
\item We have $$\dim H^1(T,{\rm Ad}\circ \rho)= \dim \frak{g}+i+i^*-(\dim \frak{g}^x+\dim \frak{g}^y+\dim \frak{g}^z).$$
\item If $H^1(T,{\rm Ad}\circ \rho)=0$ then $\rho$ is locally rigid as an element of $\Hom(T,\uG)$. (i.e. there exists a neighborhood of $\rho$ in which every element is obtained from $\rho$ by conjugation by an element of $\uG$.)
\item If $({\rm Ad}\circ \rho)^*$ has no (nontrivial) invariants on the dual $\frak{g}^*$ of $\frak{g}$, then $i=0$ and $\rho$ is  a nonsingular point  of $\Hom(T,\uG)$.
\end{enumerate}
\end{thm}

\begin{cor}\label{c:goingmaxtofull}
Let $T=T_{a,b,c}$ be a hyperbolic triangle group and $\uG$ be a simple algebraic group over $\mathbb{C}$.
 Suppose  $\rho_0: T \rightarrow \uG$ is such that ${\rm Ad}\circ\rho_0$   has  no invariants on the Lie algebra $\frak{g}$ of $\uG$ and $\rho: T \rightarrow \uG$ is such that  the closure of its image is a maximal subgroup of $\uG$ and has finite center, or is $\uG$. Then the following assertions hold:
\begin{enumerate}[(i)]
\item The representations $\rho_0$ and $\rho$ are nonsingular  in $\Hom(T,\uG)$, and  ${\rm Ad}\circ \rho_0$  and its dual (respectively,  ${\rm Ad \circ \rho}$ and its dual) have no invariants on $\frak{g}$ and $\frak{g}^*$, respectively.
\item If furthermore $\rho$ is in the irreducible component of $\Hom(T,\uG)$ containing $\rho_0$, then  $$\dim H^1(T,{\rm Ad}\circ \rho\mid_{\frak{g}})=\dim H^1(T,{\rm Ad}\circ \rho_0\mid_{\frak{g}}).$$
\end{enumerate}
\end{cor}

\begin{proof}
Let $\uH$  be the closure of the image of $\rho:T \rightarrow \uG$.
If $\uH$ is a  maximal subgroup of $\uG$ with finite center, $Z_{\uG}(\uH)\uH$ must equal $\uH$ which means that $Z_{\uG}(\uH)=Z(\uH)$ is finite. Also since $\uG$ is simple, $Z_{\uG}(\uG)$ is also finite.  It follows that ${\rm Ad}\circ \rho$ has no invariants on $\frak{g}$.  As  the adjoint representation of a simple group in characteristic zero is self-dual, we deduce that $({\rm Ad}\circ \rho)^*$ has no invariants on $\frak{g}^*$. Hence  by Theorem \ref{t:weil}(iv), $\rho$ is nonsingular and by Theorem \ref{t:weil}(ii)
$$ \dim H^1(T,{\rm Ad}\circ \rho\mid_{\frak{g}})=\dim \frak{g} - (\dim \frak{g}^{{\rm Ad}\circ\rho(x)}+ \dim \frak{g}^{{\rm Ad}\circ\rho(y)}+\dim \frak{g}^{ {\rm Ad}\circ\rho(z)}).$$
On the other hand,  ${\rm Ad}\circ\rho_0$ is also self-dual (since $\uG$ is simple and defined over $\C$). Moreover  by assumption it  has  no invariants on $\frak{g}$, and so its dual has no invariants on $\frak{g}^*$. Hence, again by Theorem \ref{t:weil},  $\rho_0$ is nonsingular and
$$ \dim H^1(T,{\rm Ad}\circ \rho_0\mid_{\frak{g}})=\dim \frak{g}- (\dim \frak{g}^{{\rm Ad}\circ\rho_0(x)}+\dim \frak{g}^{{\rm Ad}\circ\rho_0(y)}+\dim \frak{g}^{{\rm Ad}\circ\rho_0(z)}).$$
Since the restrictions of two representations in a common irreducible component of $\Hom(T,\uG)$ 
to a cyclic subgroup of $T$ are conjugate, we get
 $\dim \frak{g}^{{\rm Ad}\circ\rho(x)}=\dim \frak{g}^{{\rm Ad}\circ\rho_0(x)}$ (and similarly for $y$ and $z$); this yields the result.
\end{proof}

For a natural number $m$, we let $\delta_m^{\uG}$ denote the dimension of the subvariety $\uG_{[m]}$ of $\uG$ consisting of elements of order dividing $m$. Since $\uG$ is defined over $\mathbb{C}$, we have (with the notation of Theorem \ref{t:weil})

\begin{equation}\label{e:lietogroupfull}
\codi \frak{g}^x \leq \delta_a^{\uG}, \quad \codi \frak{g}^y \leq \delta_b^{\uG} \quad \textrm{and} \quad \codi \frak{g}^z \leq \delta_c^{\uG}.
\end{equation}


In \cite[Theorem 4.1]{LLM1} we gave the following criterion for $T$ to be saturated with finite quotients of type $X$.

\begin{thm}\label{t:key}
The hyperbolic triangle group $T$ is saturated with finite quotients of type $X$ if and only if there exist a  simple algebraic group $\uG$ over $\C$ of type $X$ and a  Zariski dense representation $\rho$ in $\Hom(T,\uG)$ which is not locally rigid (i.e. $\dim H^1(T,{\rm Ad}\circ\rho)>0$).
\end{thm}

Recall the definition of the principal homomorphism. For every simple algebraic group $\uG$ over $\mathbb{C}$, there is, up to conjugation, a unique homomorphism ${\rm SL}_2\rightarrow \uG$---called the principal homomorphism---sending every nontrivial unipotent to a regular unipotent. The induced homomorphism ${\rm SL}_2\rightarrow {\rm Ad}(\uG)$ factors through ${\rm PGL}_2$.  Since $T$ is Zariski dense in ${\rm PGL}_2(\mathbb{C})$, if $\uG$ is of adjoint type we get an induced representation $\rho_0^{\uG}:T \rightarrow {\rm PGL}_2\rightarrow \uG$.  
 
The following result is given in \cite[\S2]{LLM1}.

\begin{lem}\label{l:prinval}
 Let $\uG=X(\mathbb{C})$ be a simple adjoint algebraic group  over $\mathbb{C}$ of type $X$ and rank $r$, and $\rho_0^{\uG}: T \rightarrow \uG$ be the representation induced from the principal homomorphism ${\rm PGL}_2 \rightarrow \uG$. Write $n_1=a$, $n_2=b$ and $n_3=c$.
The following assertions hold:
\begin{enumerate}[(i)]
\item The spaces of invariants of ${\rm Ad}\circ \rho_0^{\uG}$ (on $\frak{g}$) and $({\rm Ad}\circ \rho_0^{\uG})^*$ (on $\frak{g}^*$) are trivial.
\item For $x,y,z$ acting on $\frak{g}$ via ${\rm Ad}\circ \rho_0^{\uG}$, we have $$\dim \frak{g}^x= \sum_{j=1}^r 1+2 \left\lfloor\frac{e_j}{n_1}\right\rfloor, \quad  \frak{g}^y=\sum_{j=1}^r 1+2 \left\lfloor\frac{e_j}{n_2}\right\rfloor \quad  \textrm{and} \quad \frak{g}^z= \sum_{j=1}^r 1+2 \left\lfloor\frac{e_j}{n_3}\right\rfloor.$$
where  $e_1,\dots, e_r$ are the exponents of $\uG$.
\item In particular, $$\dim H^1(T,{\rm Ad}\circ \rho_0^{\uG})=\dim \uG - \sum_{k=1}^3\sum_{j=1}^r \left(1+2\left\lfloor \frac{e_j}{n_k}\right\rfloor \right).$$

 \end{enumerate}
\end{lem}

\begin{rmk}\label{r:exp}
Recall (\cite[Planches]{Bourbaki}) that the exponents of the different root systems are as follows:
\begin{equation*}A_r: 1,2,\dots, r; \ B_r,C_r: 1,3,\dots,2r-1;\ D_r: 1,3,\dots,2r-3, r-1; \ E_6:1,4,5,7,8,11;
\end{equation*}
\begin{equation*}
E_7: 1,5,7,9,11,13, 17;\ E_8:1,7,11,13,17,19,23,29;\ F_4:1,5,7,11 ;\ G_2: 1,5.
\end{equation*}
\end{rmk}

Given a simple algebraic group $\uG$ over $\mathbb{C}$, we often obtain a Zariski dense representation $T \rightarrow \uG$ by deforming a representation in $\Hom(T,\uG)$ whose Zariski closure is a maximal subgroup of $\uG$.
More generally, we let $\Gamma$ be a finitely generated group and let $\Epi(\Gamma,\uG)$ denote the Zariski closure in the homomorphism variety $\Hom(\Gamma,\uG)$ of the set of homomorphisms $\rho\colon \Gamma\to\uG(\C)$
such that $\rho(\Gamma)$ is Zariski dense in $\uG$.
We have the following theorem:

\begin{thm}\label{t:ll}  Let $\Gamma$ be a finitely generated group, $\uG$ be a quasisimple algebraic group over $\mathbb{C}$, $\rho_0: \Gamma \rightarrow \uG(\mathbb{C})$ and $\uH$ be  the Zariski closure of $\rho_0(\Gamma)$. Assume
\begin{enumerate}[(a)]
\item $\uH$ is semisimple and connected.
\item $\uH$ is a maximal subgroup of $\uG$.
\item If $\frakg$ is the Lie algebra of $\uG$ (where the action is via ${\rm Ad}\circ \rho_0$), then
$$\dim \Epi(\Gamma,\uH) - \dim \uH < \dim Z^1(\Gamma,\frak{g}) - \dim \uG.$$
\item $\rho_0$ is a nonsingular point of $\Hom(\Gamma,\uH)$ and of $\Hom(\Gamma,\uG)$.
\end{enumerate}
Then $\Hom(\Gamma,\uG)$ has an irreducible component containing $\rho_0$ of dimension  $\dim H^1(\Gamma, \mathfrak{g})+\dim \uG$ with a nonsingular point $\rho$ on it which has a dense image. In particular, we also have $\dim H^1(\Gamma, {\rm Ad}\circ \rho)=\dim H^1(\Gamma, {\rm Ad}\circ \rho_0)$.
\end{thm}

\begin{proof}
As $\rho_0$ is a nonsingular point of $\Hom(\Gamma,\uG)$, it belongs to a unique component $\uW$ of the homomorphism variety, and
$$\dim \uW = \dim Z^1(\Gamma,\frakg) = \dim H^1(\Gamma,\frakg) + \dim \uG - \dim Z_{\uG}(\rho_0(\Gamma)) = \dim H^1(\Gamma,\frakg) + \dim \uG.$$
By a result of Breuillard, Guralnick and Larsen \cite{BGL}, the Zariski closure of the image of the representation of $\Gamma$ associated to the generic point of $\uW$
must contain a subgroup isomorphic to $\uH$.  As $\uH$ is a maximal subgroup of $\uG$, this subgroup is
isomorphic either to $\uH$ or to $\uG$.

By Richardson's rigidity theorem \cite{Richardson}, up to conjugation, there are finitely many injective homomorphisms $\uH\to \uG$.
Let $\iota_1,\ldots,\iota_k\colon \uH\to \uG$ be injective homomorphisms representing these classes.
Let $\uY = \uY_1$ denote the unique irreducible component of $\Hom(\Gamma,\uH)$ which contains $\rho_0$, and let $\uY_2,\ldots,\uY_m$ be the other irreducible components.
For each component $\uY_i$ and each injection $\iota_j$, define the conjugation map $\chi_{i,j}\colon \uG\times \uY_i\to \Hom(\Gamma,\uG)$ by
$$\chi_{i,j}(g,\rho) = g(\iota_j\circ \rho)g^{-1}.$$
The fibers of this morphism have dimension at least $\dim \uH$.  Indeed, the action of $\uH$
on $\uG\times \uY_i$ given by
$$h.(g,\rho) = (g\iota_j(h)^{-1},h\rho h^{-1})$$
is free, and $\chi_{i,j}$ is constant on the orbits of the action.  Thus, the closure of the image of $\chi_{i,j}$ has dimension at most
$\dim \uY_i+\dim \uG - \dim \uH$.  If $\uY_i$ is contained in $\Epi(\Gamma,\uH)$, then by hypothesis, this dimension is less than
$\dim Z^1(\Gamma,\frakg)$, which, in turn, is $\le \dim \Hom(\Gamma,\uG)$, since $\rho_0$ is a nonsingular point of $\Hom(\Gamma,\uG)$.
It follows that the image of $\chi_{i,j}$ is not dense in $\uW$.  As the closure of the
representation of $\Gamma$ associated to the generic point has image isomorphic to $\uH$ or $\uG$, the image of $\chi_{i,j}$ cannot be dense in
$\uW$ if $\uY_i$ is not contained in $\Epi(\Gamma,\uH)$.

Thus, the generic point of $\uW$ gives a Zariski dense homomorphism $\rho\colon \Gamma\to \uG(K)$ where $K$ is some finitely generated extension of $\C$.
This is a nonsingular point of the component $\uW$ since $\uW$ has a nonsingular point $\rho_0$.
Replacing $K$ by an algebraic closure, we may assume that it is an algebraically closed field of characteristic zero whose transcendence degree over $\Q$
is the cardinality of the continuum.  Thus, $K\cong \C$.  Fixing an isomorphism, we may take $K=\C$.  Thus, we have a nonsingular element (which we still denote $\rho$) of $\uW(\C)$ which is
a nonsingular point of this variety.   We conclude that $\dim \uW = \dim Z^1(\Gamma,\frakg)$, where $\Gamma$ acts on $\frakg$ through $\Ad\circ\rho$.  It follows that
$$\dim H^1(\Gamma,\Ad\circ\rho) = \dim \uW - \dim \uG = \dim H^1(\Gamma,\Ad\circ\rho_0).$$
\end{proof}

\begin{cor}\label{c:lltriangle}
Let $T=T_{a,b,c}$ be a hyperbolic triangle group, $\uG$ be a simple algebraic group over $\mathbb{C}$, $\rho_0: T \rightarrow \uG(\mathbb{C})$ and $\uH$ be the Zariski closure of $\rho_0(T)$. Assume
\begin{enumerate}[(a)]
\item $\uH$ is semisimple and connected.
\item $\uH$ is a maximal subgroup of $\uG$.
\item If $\frakg$ is the Lie algebra of $\uG$ (where the action is via ${\rm Ad}\circ \rho_0$), then
$$\dim \Epi(T,\uH) - \dim \uH < \dim Z^1(T,\frak{g}) - \dim \uG.$$
\end{enumerate}
Then the following assertions hold:
\begin{enumerate}[(i)]
\item $\rho_0$ is a nonsingular point of $\Hom(T,\uH)$ and $\Hom(T,\uG)$.
\item $\Hom(T,\uG)$ has an irreducible component containing $\rho_0$ of dimension  $\dim H^1(T, \mathfrak{g})+\dim \uG$ with a nonsingular point $\rho$ on it which has a dense image.
\item  $\dim H^1(T, {\rm Ad}\circ \rho)=\dim H^1(T, {\rm Ad}\circ \rho_0)$.
\end{enumerate}
\end{cor}

\begin{proof}
The first part follows from Corollary \ref{c:goingmaxtofull}(i) and Theorem \ref{t:ll} yields the second and third parts. Alternatively one could use Corollary \ref{c:goingmaxtofull}(ii) to derive the final part.
\end{proof}

 It is interesting to compare Corollary \ref{c:lltriangle}  with \cite[Theorem 5.1]{LLM1}.  There as $\uG$ was a real compact form and $\uH$ a closed subgroup, Corollary \ref{c:lltriangle} had a stronger form where we only  had  to  consider $ Z^1(T,{\rm Ad}\circ \rho_0\mid_{\frak{h}})$ and its dimension,  while here we need to work with ${\rm Epi}(T,\uH)$ which \emph{a priori} can be of higher dimension. In what follows we will show that in our special circumstances, by taking $\rho_0$ to be the representation induced from the principal homomorphism, 
$\dim {\rm Epi}(T,\uH)$
is not really larger.


\begin{prop}\label{p:epipri}
Let $T=T_{a,b,c}$ be a hyperbolic triangle group and $\uG$ be a simple adjoint algebraic group  over $\mathbb{C}$.
Let $\rho_0^{\uG}: T\rightarrow {\rm PGL}_2 \rightarrow \uG$ be the representation induced from the principal homomorphism ${\rm PGL}_2\rightarrow \uG$ and consider the action ${\rm Ad}\circ \rho_0^{\uG}$  on the Lie algebra $\mathfrak{g}$ of $\uG$ .
Then
$$ \dim {\rm Epi}(T,\uG) \leq \dim Z^1(T,{\rm Ad}\circ\rho_0^{\uG}).$$ Equivalently,
$$ \dim {\rm Epi}(T,\uG)-\dim \uG \leq \dim H^1(T,{\rm Ad}\circ\rho_0^{\uG}).$$
\end{prop}

The main ingredient in the proof of Proposition \ref{p:epipri} is the following lemma together with Theorem \ref{t:weil}(ii).

\begin{lem}\label{l:fixprifixgen}
Let $T=T_{n_1,n_2,n_3}=\langle x_1, x_2, x_3: {x_1}^{n_1}={x_2}^{n_2}={x_3}^{n_3}=x_1x_2x_3=1 \rangle$ be a hyperbolic triangle group and $\uG$ be an adjoint simple algebraic group  over $\C$ of rank $r$.
Let $\rho_0^{\uG}: T \rightarrow {\rm PGL}_2\rightarrow \uG$ be the representation induced from the principal homomorphism ${\rm PGL}_2 \rightarrow \uG$ and consider the action  ${\rm Ad} \circ \rho_0^{\uG}$ on the Lie algebra $\mathfrak{g}$ of $\uG$. Then, for $1\leq i \leq 3$,
$$\dim \frak{g}^{x_i}= \codi \uG_{[n_i]},$$
where $G_{[n_i]}$ is the subvariety of $\uG$ consisting of elements of order dividing $n_i$.
\end{lem}

\begin{rmk}\label{r:fixprifixgen}
Note that $\codi \uG_{[n_i]}$ is the minimal dimension of a centralizer of an element of $\uG$ of order dividing $n_i$ and its value is given in \cite{Lawther}.
\end{rmk}

\begin{proof}
Write $a=n_i$ and $x=x_i$. By Lemma \ref{l:prinval}(ii)
$$\dim \frak{g}^x=r+2\sum_{j=1}^r \left \lfloor \frac{e_j}{a} \right \rfloor$$ where $e_1,\dots,e_r$ are the exponents of $\uG$ which are given in Remark \ref{r:exp}. Hence the result will follow once we show that
\begin{equation}\label{e:codipri}
r+2\sum_{j=1}^r \left \lfloor \frac{e_j}{a} \right \rfloor = \codi \uG_{[a]}.
\end{equation}
We let $h=|\Phi|/r$ be the Coxeter number of $\uG$ where $\Phi$ denotes the root system of $\uG$. Suppose first that $\uG$ is of exceptional type.
If $a \geq h$, it follows immediately from Remark \ref{r:exp} that $\dim \frak{g}^x=r$ and  by Lawther \cite{Lawther} we have $\codi \uG_{[a]}=r$ and so (\ref{e:codipri}) holds. Finally if $a <h$, then \cite{Lawther} gives the value for $\codi \uG_{[a]}$ which is easily checked to be equal to $\dim \frak{g}^x$, again using Lemma \ref{l:prinval}(ii).   
 

Suppose now that $\uG$ is of classical type. We prove that (\ref{e:codipri}) holds by induction on $r$. We let
$\uG_r=\uG$, $\frak{g}_r=\frak{g}$, $h_r=h$, $L_{r,a}=\dim \frak{g}_r^x$ and $R_{r,a}=\codi {\uG_r}_{[a]}$. Letting $r_0=1,2,2$ or $4$ according respectively as $\uG=A_r$, $B_r$, $C_r$ and $D_r$, we note that (\ref{e:codipri}) holds provided that $L_{r_0,a}=R_{r_0,a}$ and $L_{r+1,a}-L_{r,a}=R_{r+1,a}-R_{r,a}$ for all $r \geq r_0$.  
 
 Write $h_r=\alpha_ra+\beta_r$ where $\alpha_r\geq 0$ and $0\leq \beta_r <a$ are integers, and for an integer $\gamma$, let $\epsilon_{\gamma}=1$ if $\gamma$ is odd, otherwise $\epsilon_{\gamma}=0$.  The value of $\codi {\uG_r}_{[a]}$, given in \cite{Lawther}, depends in general on $\alpha_r$, $\beta_r$ and $a$. 
 
Suppose first that $\uG=A_r$ where $r \geq 1$. Then $h_r=r+1$. By Lemma \ref{l:prinval}(ii) and  Remark \ref{r:exp},  $L_{1,a}=1$ (recall $a>1$)
and $$L_{r+1,a}-L_{r,a}=1+2\left \lfloor \frac{r+1}{a} \right\rfloor.$$  By \cite[p. 222]{Lawther}
$$R_{r,a}=\alpha_r^2a+\beta_r(2\alpha_r+1)-1$$ and
$$ R_{r+1,a}-R_{r,a}=(\alpha_{r+1}^2-\alpha_r^2)a+\beta_{r+1}(2\alpha_{r+1}+1)-\beta_r(2\alpha_r+1).$$
Since $h_r=r+1$ and $h_{r+1}=r+2$, we have
$$\alpha_r=\left\lfloor  \frac{r+1}{a}\right\rfloor\quad \textrm{and}\quad (\alpha_{r+1},\beta_{r+1})=\left \{\begin{array}{ll} (\alpha_r,\beta_r+1) & \textrm{if} \ 0 \leq \beta_r< a-1 \\ (\alpha_r+1,0) & \textrm{if} \ \beta_r=a-1. \end{array} \right.$$
It follows that
$$R_{1,a}=1=L_{1,a}$$ and
$$R_{r+1,a}-R_{r,a}=1+2\alpha_r=1+2\left\lfloor  \frac{r+1}{a}\right\rfloor=L_{r+1,a}-L_{r,a}$$
as required. 
 
Suppose now that $\uG=B_r$ or $C_r$ where $r \geq 2$. Then $h_r=2r$. By  Lemma \ref{l:prinval}(ii)  and Remark \ref{r:exp},  $L_{2,a}=4$ if $a \in \{2,3\}$, $L_{2,a}=2$ if $a > 3$,
and $$L_{r+1,a}-L_{r,a}=1+2\left \lfloor \frac{2r+1}{a} \right\rfloor.$$  By \cite[p. 222]{Lawther}
$$R_{r,a}=\frac12(\alpha_r^2a+\beta_r(2\alpha_r+1))+\epsilon_a \left \lceil \frac{\alpha_r}2 \right \rceil$$ and
$$ R_{r+1,a}-R_{r,a}=\frac12((\alpha_{r+1}^2-\alpha_r^2)a+\beta_{r+1}(2\alpha_{r+1}+1)-\beta_r(2\alpha_r+1))+\epsilon_a\left( \left \lceil \frac{\alpha_{r+1}}2 \right \rceil-\left \lceil \frac{\alpha_r}2 \right \rceil\right).$$
Since $h_r=2r$ and $h_{r+1}=2r+2$, we have
$$\alpha_r=\left\lfloor  \frac{2r}{a}\right\rfloor\quad \textrm{and}\quad (\alpha_{r+1},\beta_{r+1})=\left \{\begin{array}{ll} (\alpha_r,\beta_r+2) & \textrm{if} \ 0 \leq \beta_r< a-2 \\ (\alpha_r+1,0) & \textrm{if} \ \beta_r=a-2\\
 (\alpha_r+1,1) & \textrm{if} \ \beta_r=a-1. \end{array} \right.$$
It follows that
$$R_{2,a}=L_{2,a}=\left\{ \begin{array}{ll} 4 & \textrm{if}\ a \in \{2,3\}\\ 2 & \textrm{if} \ a>3. \end{array}\right.$$
Also if $0 \leq \beta_r <a-2$ then $$R_{r+1,a}-R_{r,a}=1+2\alpha_r=1+2\left \lfloor \frac{2r}{a}  \right \rfloor=1+2\left \lfloor \frac{2r+1}{a}  \right \rfloor=L_{r+1,a}-L_{r,a}.$$
If $\beta_r=a-2$ then $\alpha_r$ is odd whenever $a$ is odd, and it follows that  $$R_{r+1,a}-R_{r,a}=1+2\alpha_r=1+2\left \lfloor \frac{2r}{a}  \right \rfloor=1+2\left \lfloor \frac{2r+1}{a}  \right \rfloor=L_{r+1,a}-L_{r,a}.$$
Finally, if $\beta_r=a-1$ then $a$ is odd, $\alpha_r$ is even, and it follows that $$R_{r+1,a}-R_{r,a}=3+2\alpha_r=3+2\left \lfloor \frac{2r}{a}  \right \rfloor=1+2\left \lfloor \frac{2r+1}{a}  \right \rfloor=L_{r+1,a}-L_{r,a}.$$

It remains to consider the case $\uG=D_r$ where $r \geq 4$. Then $h_r=2r-2$. We also write $r= \eta_r a + \theta_r$ where $\eta_r \geq 0$ and $0\leq \theta_r<a$ are integers.  By  Lemma \ref{l:prinval}(ii) and  Remark \ref{r:exp},  $L_{4,2}=12$, $L_{4,3}=10$, $L_{4,{a}}=6$  if $a \in \{4,5\}$,  $L_{4,a}=4$ if  $a >5$, and   $$L_{r+1,a}-L_{r,a}=1+2\left ( \left \lfloor \frac{2r-1}{a} \right\rfloor+\left \lfloor \frac{r}a \right  \rfloor- \left \lfloor  \frac{r-1}{a}  \right \rfloor \right).$$ By \cite[p. 222]{Lawther}
$$R_{r,a}=\frac12(\alpha_r^2a+\beta_r(2\alpha_r+1))+\epsilon_a \left \lceil \frac{\alpha_r}2 \right \rceil +\alpha_r+1-\epsilon_{\alpha_r}$$ and
\begin{eqnarray*}
 R_{r+1,a}-R_{r,a}& = &\frac12((\alpha_{r+1}^2-\alpha_r^2)a+\beta_{r+1}(2\alpha_{r+1}+1)-\beta_r(2\alpha_r+1))+\epsilon_a\left( \left \lceil \frac{\alpha_{r+1}}2 \right \rceil-\left \lceil \frac{\alpha_r}2 \right \rceil\right)\\
& &+\alpha_{r+1}-\alpha_r-\epsilon_{\alpha_{r+1}}+\epsilon_{\alpha_r}.
\end{eqnarray*}
Since $h_r=2r-2$ and $h_{r+1}=2r$, we have
$$\alpha_r=\left\lfloor  \frac{2r-2}{a}\right\rfloor\quad \textrm{and}\quad (\alpha_{r+1},\beta_{r+1})=\left \{\begin{array}{ll} (\alpha_r,\beta_r+2) & \textrm{if} \ 0 \leq \beta_r< a-2 \\ (\alpha_r+1,0) & \textrm{if} \ \beta_r=a-2\\
 (\alpha_r+1,1) & \textrm{if} \ \beta_r=a-1. \end{array} \right.$$
It follows that
$$R_{4,a}=L_{4,a}=\left\{ \begin{array}{ll}
12 & \textrm{if} \ a=2\\
10& \textrm{if} \ a=3\\
6 & \textrm{if} \ a \in \{4,5\}\\
4& \textrm{if} \ a > 5.
\end{array}\right.$$

Suppose $0 \leq \beta_r <a-2$. Then  $$R_{r+1,a}-R_{r,a}=1+2\alpha_r=1+2\left \lfloor \frac{2r-2}{a}  \right \rfloor=1+2\left \lfloor \frac{2r-1}{a}  \right \rfloor.$$ Hence to show that  $R_{r+1,a}-R_{r,a}=L_{r+1,a}-L_{r,a}$ we need to check that
$$  \left \lfloor \frac{r}a \right  \rfloor- \left \lfloor  \frac{r-1}{a}  \right \rfloor=0.$$ Assume otherwise.  Writing $r-1=\eta_{r-1}a+\theta_{r-1}$ and $r=\eta_r a+\theta_r$ as above, we get
$\theta_{r-1}=a-1$, $\theta_r=0$ and $\eta_{r}=\eta_{r-1}+1$. Since $h_r=2(r-1)$, it follows that $\alpha_r=2\eta_{r-1}+1$ and  $\beta_r=a-2$. This yields $\alpha_{r+1}=2\eta_{r-1}+2=\alpha_r+1$, contradicting $\alpha_{r+1}=\alpha_{r}$. 
 
Suppose $\beta_r=a-2$. Note that $a$ is even if $\alpha_r$ is even. Now
$$R_{r+1,a}-R_{r,a}=\left \{ \begin{array}{ll} 3+2\alpha_r& \textrm{if} \ \alpha_r \ \textrm{is odd}\\
1+ 2\alpha_r & \textrm{if} \ \alpha_r \ \textrm{is even} \end{array} \right.  =1+2\epsilon_{\alpha_r}+2 \left\lfloor \frac{2r-1}{a} \right\rfloor.$$

 Hence to show that  $R_{r+1,a}-R_{r,a}=L_{r+1,a}-L_{r,a}$ we need to check that
\begin{equation}\label{e:dcaselawther}  \left \lfloor \frac{r}a \right  \rfloor- \left \lfloor  \frac{r-1}{a}  \right \rfloor=\epsilon_{\alpha_r}.\end{equation} Write $r-1=\eta_{r-1}a+\theta_{r-1}$ and $r=\eta_r a+\theta_r$ as above.  Suppose first that $\alpha_r$ is odd. Then $2\theta_{r-1}\geq a$ and $\alpha_r=2\eta_{r-1}+1$ and $\beta_r=2\theta_{r-1}-a$. Since $\beta_r=a-2$, we get $\theta_{r-1}=a-1$ which yields $\eta_r=\eta_{r-1}+1$ and so (\ref{e:dcaselawther}) holds. Suppose now that $\alpha_r$ is even so that $a$ is also even. Assume (\ref{e:dcaselawther}) does not hold. Then  $\theta_{r-1}=a-1$, $\theta_r=0$ and $\eta_{r}=\eta_{r-1}+1$. Since $h_r=2(r-1)$, it follows that $\alpha_r=2\eta_{r-1}+1$, contradicting $\alpha_{r}$ is even. 
 

Suppose $\beta_r=a-1$. Note that $\alpha_r$ is even and $a$ is odd.  Also
$$R_{r+1,a}-R_{r,a}=3+2\alpha_r= 3+2\left \lfloor \frac{2r-2}{a} \right \rfloor=1+2\left \lfloor \frac{2r-1}{a} \right \rfloor.$$

 Hence to show that  $R_{r+1,a}-R_{r,a}=L_{r+1,a}-L_{r,a}$ we need to check that
$$ \left \lfloor \frac{r}a \right  \rfloor- \left \lfloor  \frac{r-1}{a}  \right \rfloor=0.$$  Assume otherwise, and  write $r-1=\eta_{r-1}a+\theta_{r-1}$ and $r=\eta_r a+\theta_r$ as above. Then  $\theta_{r-1}=a-1$, $\theta_r=0$ and $\eta_{r}=\eta_{r-1}+1$. Since $h_r=2(r-1)$, it follows that $\alpha_r=2\eta_{r-1}+1$, contradicting $\alpha_{r}$ is even.\end{proof}

\noindent {\textit{Proof of Proposition \ref{p:epipri}}.}
Note that $$\dim {\rm Epi}(T,\uG) \leq {\rm max} \ \{\dim Z^1(T, {\rm Ad}\circ \rho): \rho \in {\rm Hom}(T,\uG),\  \overline{\rho(T)}=\uG\}.$$
Since $\uG$ is simple and defined over $\C$, a Zariski dense representation in ${\rm Hom}(T,\uG)$ composed with the adjoint representation  has no invariants on $\frak{g}$ and is self-dual. Hence by  Theorem \ref{t:weil}(i) and (\ref{e:lietogroupfull})
$$\dim {\rm Epi}(T,\uG)\leq 2\dim \uG - (\codi \uG_{[a]}+\codi \uG_{[b]}+\codi \uG_{[c]}).$$

On the other hand, since $\rho_0^{\uG}:T\rightarrow \uG$ is the representation induced from the principal homomorphism ${\rm PGL_2} \rightarrow \uG$,  ${\rm Ad}\circ \rho_0^{\uG}$  and $({\rm Ad}\circ \rho_0^{\uG})^*$ have no invariants on $\frak{g}$ and $\frak{g}^*$, respectively. Hence
by Theorem \ref{t:weil}(i)
$$\dim Z^1(T,{\rm Ad}\circ \rho_0^{\uG})=2\dim \uG-(\frak{g}^x+\frak{g}^y+\frak{g}^z).$$ The result now follows immediately from  Lemma \ref{l:fixprifixgen}. \quad $\square$

\begin{cor}\label{c:genprimet}
Let $T=T_{a,b,c}$ be a hyperbolic triangle group, $\uG$ be a simple adjoint algebraic group  over $\mathbb{C}$,  $\sigma_1: T \rightarrow \uG(\mathbb{C})$ be a homomorphism, and $\uH$ be  the Zariski closure of $\sigma_1(T)$.  Assume
\begin{enumerate}[(a)]
\item $\uH$ is semisimple and connected.
\item $\uH$ is a maximal subgroup of $\uG$.
\item The image of $\rho_0: T \rightarrow \uG$, where $\rho_0$ is the representation induced from the principal homomorphism from ${\rm PGL}_2$ into $\uG$, is inside $\uH$ (in this case $\rho_0$ is also the representation induced from the principal homomorphism from ${\rm PGL}_2$ into $\uH$) and $\rho_0$ and $\sigma_1$ belong to a common  irreducible component of $\Hom(T,\uG)$. 
\item $$\dim H^1(T,{\rm Ad}\circ \rho_0 \mid_{\frak{h}})< \dim H^1(T,{\rm Ad}\circ \rho_0\mid_{\frak{g}}).$$
\end{enumerate}
Then there exists a nonsingular representation $\rho_1: T \rightarrow \uG$ in the irreducible component of $\Hom(T,\uG)$ containing $\sigma_1$ such that $\overline{\rho_1(T)}=\uG$ and $$\dim H^1(T,{\rm Ad}\circ \rho_1\mid_{\frak{g}})=\dim H^1(T,{\rm Ad}\circ \sigma_1\mid_{\frak{g}})=\dim H^1(T,{\rm Ad}\circ \rho_0\mid_{\frak{g}}).$$
\end{cor}

\begin{proof}
The result will follow  from Corollary \ref{c:lltriangle} once we show that
\begin{equation}\label{e:genprimet}
\dim {\rm Epi}(T,\uH)-\dim \uH< \dim Z^1(T,{\rm Ad}\circ \sigma_1 \mid_{\frak{g}})-\dim \uG.
\end{equation}
Now by Proposition \ref{p:epipri} 
$$\dim {\rm Epi}(T,\uH)-\dim \uH \leq \dim H^1(T,{\rm Ad}\circ \rho_0 \mid_{\frak{h}}).$$ 
Since $\overline{\sigma_1(T)}=\uH$ is a maximal subgroup of $\uG$ and $Z(\uH)$ is finite, Corollary \ref{c:goingmaxtofull}(i) shows that   ${\rm Ad}\circ\sigma_1$ and $({\rm Ad}\circ \sigma_1)^*$ have no invariants on $\frak{g}$ and $\frak{g}^*$, respectively.  In particular (see Theorem \ref{t:weil}), we get $$\dim Z^1(T,{\rm Ad}\circ \sigma_1 \mid_{\frak{g}})-\dim \uG=\dim H^1(T,{\rm Ad}\circ \sigma_1 \mid_{\frak{g}}).$$
Now as $\sigma_1$ and $\rho_0$ are in a common irreducible component of $\Hom(T,\uG)$, Corollary \ref{c:goingmaxtofull}(ii) yields
$$\dim H^1(T,{\rm Ad}\circ \sigma_1 \mid_{\frak{g}})=   \dim H^1(T,{\rm Ad}\circ \rho_0 \mid_{\frak{g}})$$ and so
$$\dim Z^1(T,{\rm Ad}\circ \sigma_1 \mid_{\frak{g}})-\dim \uG=\dim H^1(T,{\rm Ad}\circ \rho_0 \mid_{\frak{g}}).$$
Inequality (\ref{e:genprimet}) now follows from  the assumption that
$$ \dim H^1(T,{\rm Ad}\circ \rho_0 \mid_{\frak{h}}) < \dim H^1(T,{\rm Ad}\circ \rho_0 \mid_{\frak{g}}).$$
\end{proof}

\section{Non ${\rm SO}(3)$-dense hyperbolic triangle groups}\label{s:nonso3dense}

By \cite{LL} every hyperbolic triangle group $T$ is ${\rm SO}(3)$-dense, unless $T$ belongs to $$S=\{T_{2,4,6}, T_{2,6,6}, ,T_{2,6,10},T_{3,4,4},T_{3,6,6},T_{4,6,12}\}.$$ 
The arguments in \cite{LLM1} for proving saturation break down completely for $T\in S$.
In this section we deal with these cases, proving that with a few exceptions $(T,X)$ consisting of $T\in S$ and 
$X$ an irreducible Dynkin diagram, $T$ is generally saturated with finite quotients of type $X$.

We let $\uG=X(\mathbb{C})$ be a simple adjoint algebraic group over $\mathbb{C}$ of type $X$ and  $\rho_0: T\rightarrow \uG$ be the representation induced from the principal homomorphism  ${\rm PGL}_2 \rightarrow \uG$.
To avoid confusion we will sometimes write
$\rho_0^{\uG}$ instead of $\rho_0$.
If $\frak{g}$ denotes the Lie algebra of $\uG$, we will for conciseness  write $\mathfrak{g}$ for $({\rm Ad} \circ \rho_0^{\uG}\mid_{\frak{g}})$, i.e. the action of $T$ on $\frak{g}$ via ${\rm Ad}\circ\rho_0^{\uG}$.

\begin{prop}\label{p:nonso3dense}
Let $T=T_{a,b,c}$ be a non ${\rm SO}(3)$-dense hyperbolic triangle group (i.e. $T \in S$) and $X$ be an irreducible Dynkin diagram. Then $T$ is saturated with finite quotients of type $X$ except possibly if
$(T,X)$ is as in Table \ref{tab:nso3densepe} below.
\begin{table}[h]
\caption{{\small{Non ${\rm SO}(3)$-dense  $T_{a,b,c}$ possibly not saturated with finite quotients of type $X$}}}\label{tab:nso3densepe}
\begin{tabular}{|l|l|}
\hline
$X$ & $(a,b,c)$   \\
\hline
$A_r$, $r \leq 9$ & $(2,4,6)$\\
$A_2$, $A_3$ & $(2,4,6)$, $(2,6,6)$, $(2,6,10)$ \\
$A_1$ & $(2,4,6)$, $(2,6,6)$, $(2,6,10)$, $(3,4,4)$, $(3,6,6)$, $(4,6,12)$ \\
\hline
$D_r$, $r\in \{5,7,9,13\}$ & $(2,4,6)$ \\
$D_7$ & $(2,6,6)$  \\
$D_5$ & $(3,4,4)$\\
\hline
$E_6$ & $(2,4,6)$\\
\hline
\end{tabular}
\end{table}
\end{prop}

\begin{proof}
Let $\uG$ be the simple adjoint algebraic group of type $X$ over $\mathbb{C}$. Note that as $T$ is locally rigid in ${\rm PGL}_2(\C)$ (see \cite{LLM1}), $T$ is not saturated with finite quotients of type $A_1$. We therefore assume that $r>1$ if $X=A_r$ and    divide the proof into three parts (in the spirit of \cite[Theorems 5.3, 5.5, 5.8 and 5.9]{LLM1}). 
 
Suppose first that $X = A_2$, $B_r$ ($r \geq 4$), $C_r$ ($r \geq 2$), $G_2$, $F_4$, $E_7$ or $E_8$. By Dynkin (see \cite{D2} and \cite{D3})
%
%
the image of the principal homomorphism ${\rm PGL}_2\rightarrow \uG$ is maximal in $\uG$. Let $\uH$ be the Zariski closure of $\rho_0^{\uG}(T)$. Since $T$ is Zariski dense in ${\rm PGL}_2(\C)$, $\uH \cong A_1$ is a maximal subgroup of $\uG$ and note that $\rho_0^{\uH}={\rho_0}^{\uG}$. It now follows from Corollary \ref{c:genprimet} that there is a nonsingular Zariski dense representation $\rho_1: T \rightarrow \uG$, except possibly if $\dim H^1(T,{\rm Ad}\circ \rho_0^{\uH}\mid_{\frak{h}})=\dim H^1(T,{\rm Ad}\circ \rho_0^{\uG}\mid_{\frak{g}})$. Now by \cite[Lemma 2.4]{LLM1}, $\dim H^1(T,{\rm Ad}\circ \rho_0^{\uH}\mid_{\frak{h}})=0$ and $\dim H^1(T,{\rm Ad}\circ \rho_0^{\uG}\mid_{\frak{g}})>0$ unless $X=A_2$ and $a=2$. In particular, $T$ is saturated with finite quotients of type $X$, unless $X=A_2$ and $a=2$.   
 
Suppose now that $X=A_r$ ($r \geq 3$, $r \neq 6$), $B_3$,  $D_r$ ($r \geq 5$), or $E_6$. Let $\uH$ be a maximal subgroup of $\uG$ of type $Y$ where $$Y=\left\{\begin{array}{ll} B_{r/2}  & \textrm{if}  \ X=A_r \ \textrm{and} \ r \ \textrm{even}\\
C_{(r+1)/2} & \textrm{if} \ X=A_r \ \textrm{and} \ r \ \textrm{odd}\\
G_2 & \textrm{if} \ X= B_3\\
B_{r-1} & \textrm{if} \ X=D_r\\
F_4 & \textrm{if} \ X=E_6.
 \end{array}\right.$$
Let $\rho_1: T \rightarrow \uH \hookrightarrow \uG$ be the nonsingular Zariski dense representation in $\Hom(T,\uH)$ obtained in the first part above. Since $\rho_0^{\uH}=\rho_0^{\uG}$ (see \cite[Theorems A and B]{SS}), it follows from Corollary \ref{c:genprimet} that there is a nonsingular Zariski dense representation $\rho_2: T \rightarrow \uG$,  except possibly if $\dim H^1(T,{\rm Ad}\circ \rho_0^{\uH}\mid_{\frak{h}})=\dim H^1(T,{\rm Ad}\circ \rho_0^{\uG}\mid_{\frak{g}})$.
A case by case check yields $$\dim H^1(T,{\rm Ad}\circ \rho_0^{\uH}\mid_{\frak{h}})<\dim H^1(T,{\rm Ad}\circ \rho_0^{\uG}\mid_{\frak{g}})$$ unless $X=A_3$ and $a=2$, or $X=A_r$, $r \in\{4,5,7,8,9\}$ and $(a,b,c)=(2,4,6)$, or $X=D_r$, $r \in \{5,7,9,13\}$ and $(a,b,c)=(2,4,6)$, or $X=D_7$ and $(a,b,c)=(2,6,6)$,  or $X=D_5$ and $(a,b,c)=(3,4,4)$, or $X=E_6$ and $(a,b,c)=(2,4,6)$. In particular, excluding these possible exceptions, $T$ is saturated with finite quotients of type $X$.   
 
Suppose finally that $X=D_4$ or $A_6$, and let $\uH$ be a maximal subgroup of $\uG$ of type $Y=B_3$. Let $\rho_2: T \rightarrow \uH \hookrightarrow \uG$ be the nonsingular Zariski dense representation in $\Hom(T,\uH)$ obtained in the second part above. Since $\rho_0^{\uH}=\rho_0^{\uG}$ (see \cite[Theorem B]{SS}), it follows from Corollary \ref{c:genprimet} that there is a nonsingular Zariski dense representation $\rho_3: T \rightarrow \uG$,  except possibly if $\dim H^1(T,{\rm Ad}\circ \rho_0^{\uH}\mid_{\frak{h}})=\dim H^1(T,{\rm Ad}\circ \rho_0^{\uG}\mid_{\frak{g}})$.
An easy check yields $$\dim H^1(T,{\rm Ad}\circ \rho_0^{\uH}\mid_{\frak{h}})<\dim H^1(T,{\rm Ad}\circ \rho_0^{\uG}\mid_{\frak{g}})$$ unless $X=A_6$ and $(a,b,c)=(2,4,6)$. In particular, excluding this possible exception, $T$ is saturated with finite quotients of type $X$.
\end{proof}

\section{The embedding $B_k \times B_{r-k-1}< D_r$}\label{s:bibi}

We now rule out some further possible exceptions to \cite[Theorem 1.1]{LLM1} for $X=D_r$ where $r \geq 4$ using an embedding of the form $B_k \times B_{r-k-1}< D_r$. Here we will  climb in a ``two-step ladder", where the second step, this time, is not via the representation induced from the principal homomorphism. 
In the process we will use the following  result.

\begin{lem}\label{l:flieo}
Let $\uG={\rm SO}_n(\C)$  and $t$ be any semisimple element of $\uG$ of finite order. Then
$$\dim \frak{g}^{{\rm Ad}(t)}=\binom{m_1}{2}+\binom{m_{-1}}{2}+\frac12\sum_{\lambda \in \mathbb{C}\setminus\{-1,1\}}m_{\lambda}^2$$
where, for $\lambda \in \mathbb{C}$, $m_{\lambda}$ denotes the multiplicity of $\lambda$ as an eigenvalue of $t$, in the standard representation of $\uG$.
\end{lem}

\begin{proof}
Note that if $\lambda$ is an eigenvalue of $t$ with $\lambda \neq \pm1$, then $\overline{\lambda}=\lambda^{-1}$ is also an eigenvalue with the same multiplicity. The lemma now  follows from the fact that the Lie algebra $\frak{g}$  of $\uG$ is $\Lambda^2(W)$, where $W$ denotes the natural module for $\uG$.
\end{proof}




We now make the following useful observation.
Let $\uH_1$ be a simple  adjoint algebraic group over $\mathbb{C}$ of type $B_k$ where $k\geq 2, k \neq 3$, and consider $\rho_0^{\uH_1}: T \rightarrow \uH_1$,  the representation induced from the principal homomorphism ${\rm PGL}_2\rightarrow \uH_1$.
Since $k\neq 3$, the image of the principal homomorphism ${\rm PGL}_2\rightarrow \uH_1$ is a maximal subgroup of $\uH_1$ (see \cite{D2} and \cite{D3}). As $T$ is Zariski dense in ${\rm PGL}_2$, it follows that
$\overline{\rho_0^{\uH_1}(T)}\cong A_1$ is a maximal subgroup of $\uH_1$.
By \cite[Lemma 2.4]{LLM1}, $\dim H^1(T,{\rm Ad}\circ\rho_0^{\uH_1})>0$ unless $k=2$ and $b=3$. Since every representation $T\rightarrow {\rm PGL}_2$ is locally rigid, Corollary \ref{c:lltriangle} yields (if $ k>2$ or $b\neq 3$) a nonsingular Zariski dense representation $\rho_1:T \rightarrow \uH_1$
in the same irreducible component of $\Hom(T,\uH_1)$ containing $\rho_0^{\uH_1}$ and satisfying
 $$\dim H^1(T,{\rm Ad}\circ \rho_1\mid_{\frak{h}_1})= \dim H^1(T,{\rm Ad}\circ \rho_0^{\uH_1}\mid_{\frak{h}_1}).$$\\
If  $T=T_{a,b,c}$ is a hyperbolic triangle group with $b\neq 3 $ and $(a,c)\neq (2,5)$, and $\uH_1$ is a simple adjoint algebraic group over $\C$ of type $B_3$, one can consider the nonsingular Zariski dense representation $\rho_{2,\uH_1}:T \rightarrow H_1$
obtained by deforming in a two-step ladder the representation $T \rightarrow {\rm PGL}_2 \rightarrow G_2 \hookrightarrow H_1$ induced  from the principal homomorphism ${\rm PGL}_2 \rightarrow G_2$ (see \cite[Theorem 5.8]{LLM1} and Proposition \ref{p:nonso3dense} and their proofs).  Following Corollary \ref{c:genprimet}, $\rho_{2,\uH_1}$ is in the irreducible component of $\Hom(T,\uH_1)$ containing $\rho_0^{\uH_1}$ and 
$$ \dim H^1(T,{\rm Ad}\circ \rho_{2,\uH_1}\mid_{\frak{h}_1})= \dim H^1(T,{\rm Ad}\circ \rho_0^{\uH_1}\mid_{\frak{h}_1}).$$

\begin{thm}\label{t:semisimple}
Let $T=T_{a,b,c}$ be  a hyperbolic triangle group and $\uG={\rm PSO}_{2r}(\mathbb{C})$ be the simple adjoint algebraic group over $\C$ of type $X=D_r$ where $r\geq 4$. Let $\uH={\rm SO}_{2k+1}(\C)\times {\rm SO}_{2r-2k-1}(\C)<\uG$ where $1\leq k \leq \lfloor r/2\rfloor$, i.e. $\uH=\uH_1\times \uH_2$ where $\uH_1$ and $\uH_2$ are of types $B_k$ and $B_{r-k-1}$, respectively.  Suppose $r \neq 2k+1$. Furthermore if $b=3$ assume  $\{2,3\}\cap \{k, r-k-1\}=\emptyset$ and if $(a,c)=(2,5)$ assume $3 \not \in \{k, r-k-1\}$.  

Let $\rho_1: T \rightarrow \uH_1$ be the representation obtained by deforming the representation $\rho_0^{\uH_1}$ induced from the principal homomorphism ${\rm PGL}_2\rightarrow \uH_1$ if $k \not \in  \{1,3\}$ (if $k=1$, take  $\rho_1$ to be the standard representation, and if $k=3$, take $\rho_1$ to be the representation $\rho_{2,B_3}: T \rightarrow B_3$ obtained by deforming in a two-step ladder the representation $T\rightarrow {\rm PGL}_2\rightarrow G_2$ induced from the principal homomorphism ${\rm PGL}_2 \rightarrow G_2$), $\rho_2: T\rightarrow \uH_2$ be the representation obtained   by deforming the representation $\rho_0^{\uH_2}$ induced from the principal homomorphism ${\rm PGL}_2\rightarrow \uH_2$  if $k \neq 3$ (if $k=3$, take $\rho_2$ to be the representation $\rho_{2,B_3}$), 
and let $\rho=\rho_1\oplus \rho_2: T \rightarrow \uH=\uH_1\times \uH_2$.
 Then the following assertions hold:
\begin{enumerate}[(i)]
\item $\uH$ is the Zariski closure of $\rho(T)$.
\item $\rho$ is a nonsingular point of $\Hom(T,\uH)$ and $\Hom(T,\uG)$.
\item If $\dim H^1(T, \frak{h}_1)+\dim H^1(T,\frak{h}_{2}) < \dim H^1(T, {\rm Ad}\circ \rho\mid_{\frak{g}})$ then there exists a nonsingular representation $\sigma: T \rightarrow \uG$  in the same irreducible component of $\Hom(T,\uG)$ as $\rho$, with Zariski dense image and satisfying
$$\dim H^1(T,{\rm Ad}\circ \sigma\mid_{\frak{g}})= \dim H^1(T,{\rm Ad}\circ \rho \mid_{\frak{g}}).$$
\item{If $(X,(a,b,c))$ is as in Table \ref{tab:semisimple} below then $T$ is saturated with finite  quotients of type $X$.}
\end{enumerate}
\end{thm}

\begin{table}[h]
\caption{Further possible exceptions to \cite[Theorem 1.1]{LLM1} which are ruled out in Theorem \ref{t:semisimple}}\label{tab:semisimple}
\center
\begin{tabular}{|l|l|l|}
\hline
$X$& $(a,b,c)$ & $r$\\
\hline
 $D_r$ & $(2,3,7)$ & $r \in\{7,8,10,11,13,15,16,17,19,22,23,25,29,31,37,43\}$\\
$(r\geq 4)$& $(2,3,8)$& $r\in \{7,9,10,11,13,17,19,25\}$ \\
& $(2,3,9)$ & $r\in \{7,10,11,13,19\}$\\
& $(2,3,10)$ & $r\in \{7,11,13\}$\\
& $(2,3,11)$ & $r\in \{7,13\}$\\
& $(2,3,12)$ & $r\in \{7,13\}$\\
& $(2,3,c)$, $c \geq 13$ & $r=7$\\
& $(2,4,5)$ & $r \in \{4,6,7,9,11,13,17,21\}$\\
& $(2,4,6)$ & $r\in\{5,7,9,13\}$\\
& $(2,4,7)$ & $r \in \{5,9\} $\\
& $(2,4,8)$ & $r\in \{5,9\}$\\
& $(2,4,c)$, $c \geq 9$ & $r=5$\\
& $(2,5,5)$ & $r \in \{4,6,7,11\}$\\
& $(2,5,6)$ & $r  = 7$\\
& $(2,6,6)$ & $r=7$\\
& $(3,3,4)$ & $r \in\{7,10,13\}$\\
& $(3,3,5)$ & $r =7$\\
& $(3,3,6)$ & $r = 7$\\
& $(3,4,4)$& $r=5$\\
& $(4,4,4)$& $r=5$\\
\hline
\end{tabular}
\end{table}

\begin{remk}
If $(X,(a,b,c))$  with $X=D_r$ is a possible exception to \cite[Theorem 1.1]{LLM1} not excluded in Proposition \ref{p:nonso3dense} and not figuring in Table \ref{tab:semisimple}, then one cannot use Theorem \ref{t:semisimple} to exclude it.  

\end{remk}

\begin{proof}
Since $r \neq 2k+1$, $\uH$ is a maximal subgroup of $\uG$.  Indeed, $\Lie(\uG(\C))/\Lie(\uH(\C))$ is an irreducible representation of $\uH$,
namely the tensor product of the natural representations of the factors $\uH_1$ and $\uH_2$.  Therefore, any algebraic group $\uK$ intermediate between $\uH$ and $\uG$
either has the same Lie algebra as $\uH$ or the same Lie algebra as $\uG$ (in which case it equals $\uG$).   Thus, $\uK^\circ = \uH$.  As $\uH_1$ and $\uH_2$
have distinct Dynkin diagrams without non-trivial automorphisms, all automorphisms of $\uH$ are inner.  It follows that $\uK$ is contained in $\uH Z_{\uG}(\uH)$.
If $z\in \SO(2r,\C)$ lies over an element of $Z_{\uG}(\uH)(\C)$, then the commutator of $z$ with any element of 
$$\uH(\C) = \SO(2k+1,\C)\times \SO(2r-2k-1,\C)$$ 
lies in $\{\pm I\}$.
As $\uH(\C)$ is connected, this means that the commutator is always $I$.  By Schur's lemma, $z$ must be diagonal with entries
$$(\underbrace{\lambda_1,\ldots,\lambda_1}_{2k+1},\underbrace{\lambda_2,\ldots\lambda_2}_{2r-2k-1}),$$
and then $z\in \SO(2r,\C)$ implies $\lambda_1=\lambda_2 = \pm 1$.  Thus, $z$ lies over the identity in $\uG(\C)$, and $\uK = \uH$.

Since  $\uH_1$ and $\uH_2$ are the Zariski closures of $\rho_1(T)$ and $\rho_2(T)$,  respectively, $\overline{\rho(T)}$ is mapped onto both $\uH_1$ and $\uH_2$. These are non-isomorphic simple groups (since $r \neq 2k+1$), so by Goursat's lemma, $\overline{\rho(T)}=\uH$. This shows the first part.  

The second part now follows from Corollary \ref{c:goingmaxtofull}(i).

For the third part: As $\Hom(T,\uH)=\Hom(T,\uH_1)\times \Hom(T,\uH_2)$ we have $$\dim {\rm Epi}(T,\uH)\leq \dim {\rm Epi}(T,\uH_1)+\dim {\rm Epi}(T,\uH_2).$$
Now by Proposition \ref{p:epipri} 
$$ \dim {\rm Epi}(T,\uH_i)-\dim \uH_i \leq \dim H^1(T,\frak{h}_i)  \quad \textrm{for} \ i=1,2.$$ Since $\dim \uH=\dim \uH_1+\dim \uH_2$ we get
$$ \dim {\rm Epi}(T,\uH)-\dim \uH \leq  \dim H^1(T,\frak{h}_1)+ \dim H^1(T,\frak{h}_2).$$
The third part now follows immediately from Theorem \ref{t:ll}.
The final part will follow from Theorem \ref{t:key} once we show that for $(X,(a,b,c))$ as in Table \ref{tab:semisimple}, we  can find $\uH_1$ and $\uH_2$ as above, satisfying

\begin{equation}\label{e:satisfy}
\dim H^1(T, \frak{h}_1)+\dim H^1(T,\frak{h}_{2}) < \dim H^1(T, {\rm Ad}\circ \rho\mid_{\frak{g}}).\end{equation}
Note that $\dim H^1(T, \frak{h}_1)+ \dim H^1(T, \frak{h}_2)$ can be easily calculated (see Lemma \ref{l:prinval}(iii)).  Let us concentrate on the computation of $\dim H^1(T,{\rm Ad}\circ \rho \mid_{\frak{g}})$. 
We claim that 
\begin{equation}\label{e:rhosigmaz}
 \dim H^1(T,{\rm Ad}\circ \rho \mid_{\frak{g}})= \dim H^1(T,{\rm Ad}\circ \sigma_0 \mid_{\frak{g}})
\end{equation}
 where $\sigma_0=\rho_0^{\uH_1}\oplus \rho_0^{\uH_2}$. By construction $\rho$ and $\sigma_0$ are in a common irreducible component of $\Hom(T,\uH)$ and therefore in a common irreducible component of $\Hom(T,{\uG})$. Since $\overline{\rho(T)}=\uH$ is a maximal subgroup of $\uG$, the claim will follow from Corollary \ref{c:goingmaxtofull}, once we show that ${\rm Ad}\circ\sigma_0$ has no invariants on $\frak{g}$. Note that $\sigma_0(T)$  is a subgroup of $\uH=\uH_1\times \uH_2$ of type $A_1\times A_1$. Since $\sigma_0$ is the direct sum of two irreducible representations of $T$, it follows from Schur's lemma that $Z_{\uH}(\sigma_0(T))$ consists of diagonal matrices of the form $(c_1I_{2k+1},c_2I_{2r-2k-1})$ where $c_1,c_2 \in \C$ satisfy $c_1^{2k+1}=c_2^{2r-2k-1}=1$. As $\uH<\uG={\rm PSO}_{2r}(\C)$, we get $c_1=c_2=1$ and so $Z_{\uH}(\sigma_0(T))$ is trivial. As $\uH$ is a maximal subgroup of $\uG$, $Z_{\uG}(\sigma_0(T))$ is a cyclic group. It follows that  $Z_{\uG}(\sigma_0(T))$ is trivial and so ${\rm Ad}\circ\sigma_0$  has no invariants on $\frak{g}$. This establishes the claim.  \\
Theorem \ref{t:weil}(ii) and (\ref{e:rhosigmaz}) now yield
\begin{equation}\label{e:rhosigmazf}
 \dim H^1(T,{\rm Ad}\circ \rho \mid_{\frak{g}})= \dim \frak{g}-(\dim \frak{g}^{{\rm Ad}\circ \sigma_0(x)}+\dim \frak{g}^{{\rm Ad}\circ \sigma_0(y)}+\dim \frak{g}^{{\rm Ad}\circ \sigma_0(z)}).
\end{equation}
Let $r_1=k$ and $r_2=r-k-1$ be the ranks of $\uH_1$ and $\uH_2$ respectively. For $i \in \{1,2\}$, the eigenvalues of $\rho_0^{\uH_i}(x)$ are:
$$\lambda^{-2r_i}, \lambda^{-2(r_i-1)}, \dots, \lambda^{0},\dots, \lambda^{2(r_i-1)}, \lambda^{2r_i}$$ where $\lambda$ is a primitive root of unity of degree  $2a$ (and similarly for $\rho_0^{\uH_i}(y)$ and $\rho_0^{\uH_i}(z)$ with $2b$ and $2c$, respectively).\\

Hence, the eigenvalues for $\sigma_0(x)$ are (recall $r_1<r_2$):
$$1,1, \lambda^{-2},\lambda^{-2},\lambda^{2},\lambda^{2},\dots,\lambda^{-2r_1},\lambda^{-2r_1},\lambda^{2r_1},\lambda^{2r_1},\lambda^{-2(r_1+1)},\lambda^{(2r_1+1)},\dots,\lambda^{-2r_2},\lambda^{2r_2}$$
 where $\lambda$ is a primitive root of unity of degree $2a$ (and similarly for $\sigma_0(y)$ and $\sigma_0(z)$ with $2b$ and $2c$, respectively).\\

Using Lemma \ref{l:flieo}, we can easily derive 
$$\dim \frak{g}^{{\rm Ad}\circ \sigma_0(x)}, \quad \dim \frak{g}^{{\rm Ad}\circ \sigma_0(y)} \quad \textrm{and}  \quad \dim \frak{g}^{{\rm Ad}\circ \sigma_0(z)},$$
and   then  (\ref{e:rhosigmazf}) yields $\dim H^1(T,{\rm Ad}\circ \rho \mid_{\frak{g}})$. 
 
We give  in Table \ref{tab:honehtwo} below the pairs $(X,(a,b,c))$ possibly excluded in \cite[Theorem 1.1]{LLM1} or Proposition \ref{p:nonso3dense} for which there exist $\uH_1$ and $\uH_2$ satisfying (\ref{e:satisfy}). The details can be easily checked.

\begin{table}[h]
\caption{Some pairs $(X,(a,b,c))$  for which there exist $\uH_1$ and $\uH_2$ satisfying (\ref{e:satisfy})}\label{tab:honehtwo}
\begin{tabular}{|l|l|l|l|}
\hline
$X$ & $(a,b,c)$ & $\uH_1$ & $\uH_2$  \\
\hline
$D_4$ & $(2,b,5)$ & $B_1$ & $B_2$  \\
$D_5$ & $(2,4,c)$, $c \geq 6$ & $B_1$ & $B_3$ \\
& $(3,4,4)$ & $B_1$ & $B_3$ \\
& $(4,4,4)$ & $B_1$ & $B_3$  \\
$D_6$ & $(2,b,5)$ & $B_1$ & $B_4$ \\
$D_7$ & $(2,3,c)$, $c \geq 7$ & $B_1$ & $B_5$ \\
& $(3,3,c)$, $4\leq c\leq 6$ & $B_1$ & $B_5$ \\
& $(2,4,c)$, \ $c \in \{5,6\}$ & $B_2$ & $B_4$ \\
& $(2,b,c)$, \ $\{b,c\} \subseteq \{5,6\}$ & $B_2$  & $B_4$ \\
$D_8$& $(2,3,7)$ & $B_1$ & $B_{6}$ \\
$D_9$ & $(2,3,8)$ & $B_1$ & $B_7$ \\
& $(2,4,c)$, \ $5\leq c \leq 8$ & $B_2$ &$B_6$\\
$D_{10}$ & $(2,3,c)$, $7 \leq c \leq 9$ & $B_4$ & $B_5$ \\
& $(3,3,4)$ & $B_4$ & $B_5$  \\
$D_{11}$ & $(2,3,c)$, $7 \leq c \leq 10$ & $B_4$ & $B_6$  \\
&  $(2,b,5)$ & $B_4$ & $B_6$ \\
$D_{13}$ & $(2,3,c)$, \ $7 \leq c\leq 12$ & $B_5$ & $B_7$  \\
& $(2,4,c)$, \ $c\in \{5,6\}$ & $B_5$ & $B_7$ \\
& $(3,3,4)$ & $B_5$ & $B_7$ \\
$D_{15}$ & $(2,3,7)$ &  $B_6$ & $B_8$  \\
$D_{16}$ & $(2,3,7)$ & $B_7$ & $B_8$ \\
$D_{17}$ & $(2,3,c)$, \ $c\in \{7,8\}$ & $B_7$ & $B_9$ \\
& $(2,4,5)$ & $B_7$ & $B_9$  \\
$D_{19}$&  $(2,3,c)$, $7\leq c\leq 9$ & $B_8$ & $B_{10}$ \\
$D_{21}$ & $(2,4,5)$ & $B_9$ & $B_{11}$ \\
$D_{25}$ & $(2,3,c)$, \ $c\in \{7,8\}$ & $B_{11}$ & $B_{13}$ \\
$D_r$ & $(2,3,7)$ & $B_{\lfloor r/2 \rfloor-1}$ & $B_{r-\lfloor r/2 \rfloor}$   \\
$r\in\{22,23,29,31,37,43\}$ & & &  \\
\hline

\end{tabular}
\end{table}
\end{proof}



\begin{remk}
One could try to exclude some further possible exceptions to \cite[Theorem 1.1]{LLM1} or Proposition \ref{p:nonso3dense} when $X=A_{r}$ ($r$ odd) through an embedding of the type $\uH={\rm PSO}_{r+1}(\C)<{\rm PSL}_{r+1}(\C)$. Either by starting with the representation $T \rightarrow \uH$ induced from the principal homomorphism ${\rm PGL}_2\rightarrow \uH$ (if $(D_{(r+1)/2},(a,b,c))$ is not a possible exception to \cite[Theorem 5.5]{LLM1} or Proposition \ref{p:nonso3dense}), or by starting with a representation $T \rightarrow \uH$ obtained from a representation $T \rightarrow {\rm SO}_{2k+1}(\C)\times {\rm SO}_{r-2k}(\C)$ (see Theorem \ref{t:semisimple}). However, it happens that these methods do not allow us to exclude further possible exceptions to \cite[Theorem 1.1]{LLM1} or Proposition \ref{p:nonso3dense}. 
\end{remk}

\section{The alternating group method}\label{s:altm}


In this section we will use a different homomorphism $\rho_0: T\rightarrow X(\mathbb{C})$ as a starting point for the deformation space, when $X=B_r$ or $D_r$. We let $m=2r+2$ or $2r+1$ according respectively as $X=B_r$ or $D_r$. We will take a suitable homomorphism $\rho_1$ from $T$ onto ${\rm Alt}_{m}$ and then $\rho_2: {\rm Alt}_{m}\rightarrow {\rm SO}_{m-1}(\C)$, the standard embedding (i.e. the action induced on $\mathbb{C}^{m-1}$ from the natural action of ${\rm Sym}_m$ on $\mathbb{C}^m$). We will then show that $\rho_0=\rho_2\circ \rho_1$ has a nontrivial deformation space of Zariski dense representations. This can handle many of the cases $(T,X)$ where $X=B_r$  or $D_r$ (see Lemma \ref{l:c} below), but we will only bother to check and prove the cases that have not been worked out by the principal homomorphism method or by deforming a representation of the form $T\rightarrow B_{k}\times B_{r-k-1}$. \\

 

\begin{lem}\label{l:c}
Let $X=B_r$ (respectively, $D_r$) and $H={\rm Alt}_{m}$ where $m=2r+2$ (respectively, $2r+1$). Let $\rho_2$ be the standard representation of $H$ into ${\rm SO}_{m-1}(\C)$. If there exists an epimorphism $\rho_1$ from $T$ to $H$ and  $\rho_0=\rho_2\circ\rho_1$ is such that $\dim H^1(T,{\rm Ad}\circ\rho_0)>0$ then $T$ is saturated with finite quotients of type $X$.
\end{lem}

\begin{remk}
Note that if $T=T_{a,b,c}$ is saturated with finite  quotients of a given type, then so is $T_{a',b',c'}$ where $a'$, $b'$, $c'$ are any positive multiples of $a$, $b$, $c$, respectively. Indeed $T_{a,b,c}$ is a quotient of $T_{a',b',c'}$.
\end{remk}

\begin{proof}
Since $r\geq 2$, the action of $H$ on ${\rm Lie}(X)$ is irreducible  (see \cite[Ex. 4.6]{FH} and \cite[Proposition 3.1]{LL} and its proof).
As $\dim H^1(T,{\rm Ad}\circ \rho_0)>0$, $\rho_0$ has  nontrivial deformation.  We fix an irreducible 
component $\uX$ of $\Hom(T,\uG)$ containing $\rho_0$ on which the deformation is non-trivial.
Since being irreducible is an open condition, irreducibility on ${\rm Lie}(X)$ must hold in an open  neighborhood of $\rho_0$ in $\uX$. For $\rho$ in such a neighborhood, $\rho(T)$ stabilises ${\rm Lie}(\overline{\rho(T)})$.  Since $\rho(T)$ acts irreducibly on ${\rm Lie}(X)$, either ${\rm Lie}(\overline{\rho(T)})$  is  zero or 
it equals ${\rm Lie}(X)$. In the second case, by Theorem \ref{t:key}, we are done.  In the first case,  $\rho(T)$ is finite.  
From Jordan's Theorem, it then follows that  $\rho(T)$ has a normal abelian subgroup of bounded index,
or equivalently, $\rho(T_0)$ is abelian for some $T_0\subset T$ of bounded index.  As $T$ is finitely generated, there are finitely many possible $T_0$,
and their intersection $T_1$ is of finite index in $T$.  
If for all $\rho$ in $\uX(\C)$, $\rho(T)$ is of bounded order, then $\rho_0$ is locally rigid, a contradiction. If they are unbounded, then, in the generic representation of $\uX$, the Zariski closure is infinite and  
virtually abelian, again a contradiction. 
 \end{proof}


\begin{lem}\label{l:altgen}
Let $H={\rm Alt}_m$ and $(a,b,c)$ be as in Table \ref{t:one} below. Then $H$ is a quotient of $T=T_{a,b,c}$ with torsion-free kernel. Moreover, we can find elements $A$ and $B$  of $H$ of respective orders $a$ and $b$ such that $AB$ has order $c$ and  $\langle A,B \rangle=H$, where $A$, $B$ and $AB$ have cycle shapes as given in Table \ref{t:one} below.
\end{lem}

\begin{table}[h]
\caption{{\footnotesize{Pairs $(A,B)$ of elements of $H={\rm Alt}_m$ such that $H=\langle A,B\rangle$, $|A|=a$, $|B|=b$ and $|AB|=c$}}}\label{t:one}
\center{
\begin{tabular}{|l|l|l|l|l|}
\hline
$H={\rm Alt}_m$ & $(a,b,c)$ & $A$ & $B$ & $AB$\\
\hline
${\rm Alt}_8$ & $(3,3,15)$ & $(3)^2(1)^2$ & $(3)^2(1)^2$ & $(5)(3)$\\
\hline
${\rm Alt}_9$ & $(2,3,15)$  & $(2)^4(1)^1$ & $(3)^3$ & $(5)^1(3)^1(1)^1$ \\
& $(3,3,7)$ & $(3)^3$ & $(3)^3$ & $(7)^1(1)^2$\\
& $(3,3,9)$ & $(3)^3$ & $(3)^2(1)^3$ & $(9)^1$\\
& $(3,3,10)$ & $(3)^3$ & $(3)^3$ & $(5)^1(2)^2$\\
& $(3,3,12)$ & $(3)^3$ & $(3)^2(1)^3$ & $(4)^1(3)^1(2)^1$\\
& $(3,3,15)$ & $(3)^3$ & $(3)^3$ & $(5)^1(3)^1(1)^1$\\
\hline
${\rm Alt}_{11}$ & $(2,3,11)$  & $(2)^4(1)^3$ & $(3)^3(1)^2$ & $(11)^1$ \\
\hline
\end{tabular}}
\end{table}

\begin{proof}
Using MAGMA \cite{MAGMA} one can find a subgroup $S$ of $T$ of index $m$ such that the action of $T$ on the set $T/S$ of cosets of $S$ in $T$ induces a homomorphism $f: T \rightarrow {\rm Sym}(T/S)$ satisfying $f(T)={\rm Alt}_m$ and $f(x)=A$, $f(y)=B$ where $A$, $B$ are elements of ${\rm Alt}_m$ such that $A$, $B$ and $AB$ have cycle shapes given in Table \ref{t:one}. The result follows.
\end{proof}

\begin{remk}
In the proof of Lemma \ref{l:altgen}, one can give $A$ and $B$ explicitly. However, for  conciseness, we only give the cycle shapes of $A$, $B$ and $AB$. This suffices for computing $\dim H^1(T,{\rm Ad}\circ \rho_0)$ as needed below.
\end{remk}

\begin{prop}\label{p:pos}
Let $X=B_r$ (respectively, $D_r$) and $H={\rm Alt}_{m}$ where $m=2r+2$ (respectively, $2r+1$).  Suppose $(H,(a,b,c))$ appears in Table \ref{t:one}.  Let $\rho_2$ be the standard representation of $H$ into ${\rm SO}_{m-1}(\mathbb{C})$, $\rho_1$ be the epimorphism from $T=T_{a,b,c}$ to $H$  provided by Lemma \ref{l:altgen}, and  $\rho_0=\rho_2\circ\rho_1$.  Then $\dim H^1(T,{\rm Ad}\circ\rho_0)>0$  and so $T$ is saturated with finite  quotients of type $X$.
\end{prop}

\begin{proof} 
We first show that $\dim H^1(T,{\rm Ad}\circ\rho_0)>0$. Let $W$ be the natural module for ${\rm SO}_{m-1}(\mathbb{C})$  and  $V={\rm Lie}(X)=\Lambda^2(W)$. Since $H$ is irreducible on ${\rm Lie}(X)$, Theorem \ref{t:weil}(ii) yields
$$\dim H^1(T,{\rm Ad}\circ\rho_0)=\dim V-( \dim V^{{\rm Ad}\circ\rho_0(x)}+ \dim V^{{\rm Ad}\circ\rho_0(y)}+ \dim V^{{\rm Ad}\circ\rho_0(z)}).$$
Since $\dim V$ is either $r(2r-1)$ or $r(2r+1)$ according respectively as $X$ is $D_r$ or $B_r$, it now remains to compute $\dim V^{{\rm Ad}\circ\rho_0(t)}$ for $t \in \{x,y,z\}$.
Note that if $\rho_1(t)$ has cycle shape  $(1)^{n_0}(b_1)^{n_1}\dots(b_s)^{n_s}$, then $\rho_0(t)$  acts on $W$ with eigenvalues:
1 occuring with multiplicity $-1+\sum_{i=0}^s n_i$,
and  $\beta_i,\dots,\beta_i^{b_i-1}$ occuring $n_i$ times for $1\leq i \leq s$, where $\beta_i$ is a primitive $b_i$-th  root of unity.  Hence, using Lemma \ref{l:flieo}, an easy check yields $\dim H^1(T,{\rm Ad }\circ \rho_0)>0$ in all cases. The result now follows from Lemma \ref{l:c}.

\end{proof}

The following result shows that the alternating method cannot be used to determine whether $T$ is saturated with finite  quotients of type $X$ in the remaining open cases $(T,X)$ whith $X=B_r$ or $D_r$.

\begin{lem}\label{l:altnotgen}
If the pair $({\rm Alt}_m,(a,b,c))$ appears in Table \ref{t:two} below, then ${\rm Alt}_m$ is not $(a,b,c)$-generated.
\end{lem}

\begin{table}[h]
\center{
\caption{Some pairs $({\rm Alt}_m,(a,b,c))$ such that ${\rm Alt}_m$ is not $(a,b,c)$-generated}\label{t:two}
\begin{tabular}{|l|l|}
\hline
${\rm Alt}_m$ & $(a,b,c)$\\
\hline
${\rm Alt}_8$ & $(2,3,c),\quad c \geq 7$\\
& $(2,4,5)$, $(2,5,5)$\\
& $(3,3,c),  \quad c \geq 4, \quad c\not \equiv 0 \mod 15$ \\
\hline
${\rm Alt}_9$ & $(2,3,c), \quad c \geq 7, \quad c \not \equiv 0 \mod 15$\\
& $(3,3,c),  \quad c \geq 4, \quad c \not \equiv 0 \mod \alpha,  \alpha \in \{7,9,10,12,15\}$ \\
\hline
${\rm Alt}_{11}$ & $(2,3,c), \quad c\geq 7, \quad c \not \equiv 0 \mod 11$\\
& $(2,4,5)$\\
& $(3,3,4)$\\
\hline
${\rm Alt}_{19}$ & $(2,3,7)$ \\
\hline
\end{tabular}}
\end{table}

The following result   (see  \cite{CC}) is the main ingredient in proving Lemma \ref{l:altnotgen}.

\begin{lem}\label{l:u}
Suppose the group $H$ is generated by permutations $h_1$, $h_2$, $h_3$ acting on a set $\Omega$ of size $n$ such that $h_1h_2h_3$ is the identity permutation. If the generator $h_i$ has exactly $m_i$ cycles (for $1 \leq i \leq 3$) and $H$ is transitive on $\Omega$ then
$$m_1+m_2+m_3\leq n+2 \quad(\textrm{and} \quad m_1+m_2+m_3 \equiv n \mod 2).$$
\end{lem}

\noindent{\textit{Proof of Lemma \ref{l:altnotgen}}.}
Applying Lemma \ref{l:u} we immediately reduce to the case $m \in \{8,9\}$. Moreover using \cite[Theorem, pp. 84--85]{Conder} where Conder gives for $m \geq 5$ a triple $(a,b,c)$ with $1/a+1/b+1/c$ maximal such that ${\rm Alt}_m$ is an $(a,b,c)$-group, we are reduced to the following cases:
$$m=8 \quad \textrm{and} \quad(a,b,c) \in \{(3,3,6),(3,3,7))\}$$ or $$ m=9 \quad \textrm{and} \quad (a,b,c) \in \{(2,3,12),(3,3,4),(3,3,5),(3,3,6)\}.$$
It remains to show that in these cases ${\rm Alt}_m$ is not $(a,b,c)$-generated.
Using \cite{MAGMA} one easily checks that indeed this does not occur.  \quad $\square$

\end{document}